\documentclass[num-refs]{wiley-networks}

\usepackage{orcidlink}
\usepackage{nicefrac}
\usepackage{comment}
\usepackage[noend]{algorithmic}
\usepackage{algorithm}
\usepackage{enumitem}
\usepackage{booktabs} 
\usepackage{multirow}

\newcommand{\p}{k}
\newcommand{\NP}{$\mathcal{NP}$}

\newcommand{\review}[1]{{\color{black}#1}}
\newcommand{\R}{\mathbb{R}}
\newcommand{\PrimerModelo}{{\((\mathcal{M}_1)\)}}
\newcommand{\SegundoModelo}{{\((\mathcal{M}_2)\)}}
\newcommand{\TercerModelo}{{\((\mathcal{M}_3)\)}}
\newcommand{\TercerModeloRelajado}{{\((\mathcal{MR}_3)\)}}

\newcommand{\DP}{\mathsf{d_p}}
\newcommand{\LL}{\mathsf{L}}

\newcommand{\MaxT}{$\mathsf{\times}$}
\newcommand{\MaxText}{$\mathsf{T_{M}}$}
\newcommand{\LC}{\textsf{LC}}
\newcommand{\SPC}{\textsf{SPC}}
\newcommand{\LBC}{\textsf{LBC}}
\newcommand{\SC}{\textsf{S-C}}
\newcommand{\No}{\textsf{None}}
\newcommand{\CGC}{\textsf{CGC}}

\DeclareMathOperator*{\argmin}{arg\,min}

\allowdisplaybreaks

\papertype{Original Article}
\paperfield{Networks} 

\title{Optimizing Connected Components Graph Partitioning with Minimum Size Constraints using Integer Programming and Spectral Clustering Techniques}

\author[1\authfn{1}\authfn{2}]{Mishelle Cordero}
\author[2,1\authfn{1}]{Andrés Miniguano--Trujillo\orcidlink{0000-0002-0877-628X}}
\author[3,1\authfn{1}]{Diego Recalde\orcidlink{0000-0002-6981-2272}}
\author[3\authfn{1}]{Ramiro Torres\orcidlink{0000-0003-2057-1719}}
\author[3\authfn{1}]{Polo Vaca\orcidlink{0000-0003-0387-6668}}

\contrib[\authfn{1}]{Equally contributing authors.}
\affil[1]{Research Centre on Mathematical Modelling (ModeMat), Escuela Politécnica Nacional, Quito, Ecuador}
\affil[2]{Maxwell Institute for Mathematical Sciences, The University of Edinburgh and Heriot-Watt University, Bayes Centre, 47 Potterrow, Edinburgh,  United Kingdom}
\affil[3]{Escuela Politécnica Nacional, Department of Mathematics, Ladrón de Guevara E11-253, Quito, Ecuador}

\corraddress{Andrés Miniguano--Trujillo \\
Maxwell Institute for Mathematical Sciences, 
Edinburgh,  United Kingdom}
\corremail{Andres.Miniguano-Trujillo@ed.ac.uk}

\presentadd[\authfn{2}]{Escuela Superior Politécnica del Litoral, ESPOL, Guayaquil, Ecuador 
}

\fundinginfo{M.C. was partially supported by Escuela Politécnica Nacional's Master Program in Mathematical Optimization Scholarship.
A.M-T. acknowledges support of MAC-MIGS CDT Scholarship under EPSRC grant EP/S023291/1. 
D.R., R.T., and P.V. acknowledge the support by Escuela Politécnica Nacional Research Project \texttt{PII-DM-2019-03} and \texttt{PII-DM-2020-02}.}

\runningauthor{Mishelle Cordero et al.}

\begin{document}

\maketitle

\begin{abstract}
 In this work, a graph partitioning problem in a fixed number of connected components is
considered. Given an undirected graph with costs on the edges,  the problem \review{consists of partitioning} the set of nodes into a fixed number of subsets with minimum size, where each subset induces a connected subgraph with minimal edge cost.
\review{The problem naturally surges in applications where connectivity is essential, such as  cluster detection in social networks, political districting, sports team realignment, and energy distribution.}
Mixed Integer Programming formulations together with a variety of valid inequalities are demonstrated and computationally tested. 
An assisted column generation approach by spectral clustering is also proposed for this problem with additional valid inequalities. Finally, the methods are tested for several \review{simulated instances,} and computational results are discussed. 
\review{Overall, the proposed column generation technique enhanced by spectral clustering offers a promising approach to solve clustering and partitioning problems.}

\end{abstract}

\section{Introduction}
\label{intro}

Let \(G = (V,E)\) be an undirected graph, with \(V\) and \(E\) its node and edge sets, respectively. 
Let \( d: E \to \R^{+}\) be a cost function over the set of edges. Additionally, let \(k \geq 2\) and \(\alpha \geq 2\) be two integer numbers.
The aim of this work consists of finding a partition of \(V\) into \(k\) subsets, such that each subset from the partition induces a connected subgraph of \(G\) with at least \(\alpha\) nodes. Additionally,  the total cost associated \review{with} the edges with end nodes on the same subset is minimized. 


The idea for combining graph partitioning with connectivity constraints was motivated by a sub-problem that appears in \cite{Gutierrez2019} where the authors model an integrated vehicle and pollster routing problem for scheduling visits of pollsters to selected stores and routing a vehicle fleet used to transport them. Specifically, during the planning of routes for data collection, the authors solve a graph partitioning problem on a complete graph to determine the set of stores to be visited on each day. However, if the graph were non-complete, which is not a hypothesis considered by the authors in the previous application, connectivity constraints would be required in the partitioning problem in order to construct daily feasible routes. In fact, a non-complete graph is a natural choice for pedestrian and traffic networks, due to the impracticalities of connecting very distant nodes.

The graph partitioning problem has been studied since the sixties, when some seminal works can be identified.  Schurmann \cite{SCHURMANN1964} proposed a scheme for partitioning a graph in the minimum number of cycles required for estimating the number of loops passing through any program statement. Carlson and Nemhauser \cite{Nemhauser1966} partitioned a graph in at most $k$ subsets without bounds on the size of each subset using a quadratic formulation. Later, Kernighan and Lin \cite{Kernighan1970, Kernighanphdthesis} partitioned a graph into subsets of given sizes, where  they provided exact methods for problems with certain restrictions on the graph  and  a heuristic method for the \review{not-restricted} case. Moreover, Christofides and Brooker \cite{Christofides1976}  considered  the  bipartition problem minimizing the sum of the costs of the cut links, and a tree search method was used imposing an upper bound on the size of nodes. The graph partitioning problem is well-known to be {\NP}--hard and the proof of its complexity is reported in \cite{Garey1976}.

From the works cited above, it can be certainly inferred that various variants of the partitioning problem are widely reported in the literature. As such, it is important to clarify the differences among these variants. The most general problem in this context is the clustering problem which consists of finding the best partition of a set of entities into similar disjoint subsets. The clustering problem, in the most general form, does not require that the number of subsets is known \emph{a priori}; thus, a first variant of the clustering problem appears when the set of entities is bisected and the cardinalities of the  two sets obtained differ in at most one element. The bisection is naturally extended to a $k$-partitioning problem. From the latter variant, different constraints are included to define several other problems: for example, requiring a minimum cardinality on each subset,  imposing upper and lower bounds on them (size-constrained partitioning), and the same number of elements in each subset (\(k\)-way equipartition), or even different cardinalities for each subset (general partitioning problem). On the other hand, if weights are associated to each entity, then balance and capacity requirements with respect to the total weight of each subset  can be accomplished. Additionally, if a graph structure is provided over the set of entities,  a special characterization for each subset can be demanded: connectivity, clique generation, trees, etc. 

A complete survey and recent advances in this topic can be found in Buluc et al. \cite{Buluc2016}. Applications of the graph partitioning problem can be found in many areas, for instance: parallel computing \cite{HENDRICKSON_2000}, VLSI circuit design \cite{VLSI_2011}, mobile wireless communications \cite{Fairbrother_EtAl_2017}, image processing \cite{JianboShi2000}, or sports team realignment \cite{Mitchell_2003}. On the other hand, these problems have been studied with different techniques as Integer Programming, heuristic methods, semidefinite programming, or spectral methods.

Firstly, many studies using Integer Programming techniques are widely reported in the literature.
Gröstchel and Wakabayashi \cite{GROTSCHEL} studied a clustering problem that arises in qualitative data analysis. The authors identified the combinatorial problem and presented  the clique partitioning problem on a complete graph together with several facets, valid inequalities, and a Branch \& Cut scheme. Conforti, Rao, and Sassano \cite{Conforti1990a,Conforti1990b} partitioned the set of nodes into two sets such that the cardinality of these subsets differ in at most one node. The authors provided some Integer Programming formulations, {studied} the dimension of the equicut polytope, and some classes of facet-inducing inequalities were described. {In a} similar way, Rao and Chopra  \cite{Chopra1993}  described several forms of the  $k$-partitioning problem reporting several results from a polyhedral point of view. Ferreira et al. \cite{Ferreira} introduced capacity constraints on the sum of node weights on each subset of the partition and Labbé and Özsoy \cite{LABBE} reported a competitive Branch \& Cut algorithm for  the clique partitioning problem including  upper and lower bounds on the size of the cliques.  Also, a Branch \& Bound framework for solving a minimum graph bisection is presented in  \cite{Delling2014},  where lower bounds are computed heuristically. Miyauchi and  Sukegawa \cite{Miyauchi_2014} introduced a class of redundant transitivity constraints in the standard formulation for the clique partitioning problem reporting a compact formulation and a compact linear programming relaxation.  In a similar way, Koshimura et al. \cite{Koshimura_2022} introduced a series of concise Integer Linear Programming formulations that can reduce additional transitivity constraints and the authors evaluated the amount of reduction based on a simple model in which edges are chosen independently. Column generation techniques are also applied for graph partitioning. Al-Ykoob and Sherali \cite{AlYkoob2020} partitioned a complete weighted graph into cliques, where each clique has the same number of nodes, and  the total weight of the resulting subgraphs is minimized.  Moreover, Ji, and Mitchell \cite{Ji2007} considered the partitioning problem on  a complete graph,  such that  each clique has  a minimum number of nodes, and the total weight of the edges with both endpoints in the same subclique is minimized. The authors proposed a Branch \& Price method to solve the problem; they demonstrated the necessity of cutting planes for this problem and suggested effective ways of adding such planes.

Concerning a heuristic point of view, many approximation algorithms can be found in the literature.
Hendrickson and Leland \cite{Hendrickson1995} proposed CHACO, a solver that contains geometric, spectral, and multilevel heuristic algorithms for partitioning graphs.  METIS \cite{Karypis1998}  is another solver based on an efficient implementation of multilevel heuristics and some variants that are tested in different applications where graphs of large size arises: finite elements, linear programming, or transportation. Sanders and  Shulz \cite{Sanders2012} introduced a distributed evolutionary heuristic using a multilevel algorithm that combines an effective crossover with mutation operators. Jovanovic et al. \cite{Jovanovic_2016} presented an ant colony optimization algorithm combined with a local search method for solving the problem of the maximum partitioning of graphs with supply and demand. Wu et al. \cite{Wu_2019} introduced  a deterministic annealing neural network algorithm that attempts to obtain a high-quality solution by following a path of minimum points of a barrier problem. Recently, Bruglieri and Cordone \cite{Bruglieri_2021} proposed a two-level tabu search algorithm and an adaptive large neighborhood search algorithm to solve the minimum gap graph partitioning problem on instances of large-size. These methods have been tested on large-scale graphs providing good solutions in a reasonable time.
 
Semidefinite Programming is another technique that has been extensively used for solving graph partitioning problems. Alizadeh \cite{Alizadeh1995} introduced semidefinite relaxations for graph partitioning problems {together with some other} combinatorial optimization problems. Wolkowicz and Zhao \cite{Wolkowicz1999} derived a new semidefinite programming relaxation for the general graph partitioning problem obtained {through the Lagrangian dual} of a quadratic formulation of the problem. Lisser and Rendl  \cite{Lisser_2003} considered a simplified version for the telecommunication industry in which {a set of local exchange stations was partitioned} into subsets of equal cardinality using both linear and semidefinite relaxations. Hager, Phan, and Zhang \cite{Hager2011} presented an exact algorithm for graph partitioning using a Branch \& Bound method applied to a continuous quadratic programming formulation. The authors showed that the semidefinite approach generally led to much tighter lower bounds in a series of numerical experiments. Furthermore, Sotirov \cite{Sotirov2014} derived a new semidefinite programming relaxation for the general graph partitioning problem, which is based on matrix lifting {with matrix variables} having order equal to the number of vertices of the graph. A survey-type article introducing semidefinite optimization is reported in \cite{Rendl2012}.

Additionally, spectral partitioning is a family of {methods focused on finding} clusters using the eigenvectors of the Laplacian matrix derived from a set of pairwise similarities between the nodes of the graph. In the seminal paper, Fiedler \cite{Fiedler1975} showed a fundamental relation between eigenvectors of a nonnegative symmetric matrix with the degree of reducibility of some principal sub matrices, which was later applied in the theory of algebraic connectivity of undirected graphs. {A few years later}, Barnes and Hoffman \cite{Barnes1984} used linear programming in combination with graph spectral information, obtaining lower objective bounds for a graph partitioning problem. Qiu and Hancock \cite{Qiu2006} described a spectral method that can be used to decompose graphs into non-overlapping neighborhoods that can be used for the purposes of both matching and clustering. The work of Spielman and Teng \cite{Spielman2007} establishes and exploits a connection between the Fiedler value and geometric embeddings of graphs, and it provides eigenvalue bounds by proving that every planar graph has an embedding in the Euclidean space. It was also reported that spectral partitioning methods worked well on bounded-degree planar graphs and finite element meshes. In \cite{DiNardo2017}, Di Nardo et al. found the optimal layout of districts in a real water distribution system, taking into account both geometric and hydraulic features, through weighted spectral clustering methods. Finally, \cite{Meila2015} provides an overview of the main research developments in cluster analysis and spectral clustering. 

Connectivity is a common requirement in many different applications. For instance, it can be identified in cluster detection in social network analysis \cite{Moody2003}, forest harvesting \cite{Carvajal2013,Martins2005}, bioinformatics \cite{Backes2011,Huffner2014}, political districting \cite{Ricca2013}, and energy distribution \cite{Ljubi2005}.  Furthermore, graph problems imposing connectivity constraints are discussed in different works previously reported in the literature. J\"{u}nger et al. \cite{Jnger1985} considered the problem of partitioning the edge set into subsets of a prescribed size, such that the edges in each subset constitute a connected subgraph. Here, sufficient conditions for the existence of a partition  in terms of the edge connectivity of the graph are derived.  Ito et al. \cite{Ito2004} partitioned a graph into connected components by deleting edges so that the total weight of each component satisfies lower and upper bounds; the authors obtained pseudo-polynomial-time algorithms to solve the  problems for series-parallel graphs. Lari et al. \cite{Lari2015} studied the problem of partitioning into a fixed number of  connected components such that each component contains exactly one node specified previously as a center; the authors analyzed different optimization problems considering different objective functions, and they proved {\NP}--hardness for specific classes of graphs, while also providing polynomial time algorithms for additional classes. The  connected subgraph polytope of a graph and several polyhedral properties of the connectivity are studied in \cite{Wang2017}, where the authors focused on the study of  nontrivial classes of valid inequalities for the connected subgraph polytope. In recent years, Hojny et al. \cite{Hojny2020} considered the general graph partitioning problem in connected components, aiming to maximize the number of edges among components. The authors proposed  two  mixed integer linear formulations  and  new Branch \& Cut techniques considering cuts, branching rules, propagation, primal heuristics, and symmetry breaking  inequalities, and extensive numerical experiments were presented. Miyazawa et al. \cite{Miyazawa2021} proposed three mixed integer linear programming formulations for the problem of partitioning a vertex-weighted connected graph into a fixed number of  connected subgraphs with similar weights; also, polyhedral results for one formulation were included. Some results concerning the computational complexity of the partitioning problem in connected components can be found in \cite{Dyer1985}.

In this work, the graph partitioning problem in \(k\) connected components with at least \(\alpha\) nodes is addressed from an exact and a heuristic point of view. The first contribution of this paper consists of proposing two new Mixed Integer Programming (MIP) formulations for the graph partitioning problem, both considering well-known flow constraints to guarantee connectivity. The second and main contribution is the introduction of a novel Column Generation strategy which incorporates spectral clustering techniques within a third Integer Programming (IP) formulation. Several families of inequalities are proved to be valid, and they are included in the solution framework. The main drawbacks of IP techniques lie in their inability to attain an optimal solution for large-sized instances. In this regard, the effectiveness and scope of both exact and heuristic methods are demostrated with several simulated instances.

The paper is organized as follows. In Section \ref{modelo}, two MIP and one IP formulations are described. For the Column Generation approach, a heuristic to solve the associated {\NP}--hard pricing problem is proposed. Section \ref{sec:Valid_ineqs} introduces and proves several families of valid inequalities for the associated polyhedra of each formulation. Computational experiments are discussed in Section \ref{resultados}, and conclusions are given in Section \ref{conclusiones}.


\section{Notation and Formulations}
\label{modelo}

Let \(G = (V,E)\) be an undirected graph, where \(V = \{1,\ldots, n\}\) represents its node set and \(E \subset \big\{ \{i,j\}:\, i,j\in V, i \neq j \big\} \) its edge set. The set of neighbors of a set of nodes $C$ is defined by  \(N(C)=\{v \in V\setminus C: \{w,v\}\in E, w\in C\}\).  Let $\delta(S)$ be the set of edges in the cut associated with a set of nodes $S$.  Moreover, let \( d: E \to \R^{+}\) be a cost function, and let \(k,\alpha \geq 2\) be two integer numbers. Here, \(k\) denotes the fixed number of connected components to partition \(G\) into, and \(\alpha\) represents the minimum cardinality of each connected component. A \(k\)-partition of \(V\) is a set \(\{V_1,V_2,\ldots,V_k\}\) where \(V_i\cap V_j=\emptyset\) for all \(i\neq j\), \(\bigcup\limits_{c=1}^k V_c=V\), and \(V_c \neq \emptyset\) for all \( c \in [k]\), where $[k]$ denotes the set $\{1,\ldots,k\}$. { Additionally,  \( E(V_c) \) represents the subset of edges with both end nodes in \(V_c\), for any \(c \in [k] \) and \( \bar{E}=\left\{\{i,j\} \not\in E : \, i,j \in V, i \neq j  \right\} \) . }

The partitioning problem in connected components with minimum size constraints consists of finding a  $k$-partition $\{V_1,V_2,\ldots,V_k\}$ such that each {subset $V_c$ induces a connected subgraph \((V_c, E(V_c))\),  with \(|V_c| \geq \alpha\), and the total cost of edges in \(\bigcup\limits_{c\in[k]} E(V_c)\) is minimized. A challenging task of this problem lies in capturing the connectivity requirement for the graph partitioning problem. 
Such requirement limits the applicability of traditional modelling tools and classical spectral clustering techniques. To address this problem, this paper proposes the novel use of flow constraints under the scope of mixed integer programming, which can effectively evaluate the connectivity of a cluster, and a clever modification of the spectral clustering algorithm within a column generation strategy.
}

Observe that each connected component of \(G\) will have at most \(\beta := n-(k-1)\alpha\) nodes. This comes as in the worst case, if \(k-1\) connected components have the minimum cardinality \(\alpha\), then the remaining component must contain \( \beta\) nodes. Moreover,  \(\alpha\) has to be less than or equal to \( \left\lfloor \nicefrac{n}{k} \right\rfloor\), otherwise the problem is infeasible.

\subsection{Formulations via Mixed Integer Programming }

Two MIP models, labeled {\PrimerModelo} and {\SegundoModelo}, are proposed in this subsection. The formulation of these two MIPs uses ideas also explored in Hojny et al. \cite{Hojny2020} and Miyazawa et al. \citep{Miyazawa2021}, where connectivity is modeled using flow-based constraints.
\review{In \cite{Hojny2020} and \citep{Miyazawa2021}, the proposed models maximize the total weight of edges between partitions, which is equivalent to the objectives proposed in this section. 
The addition of the minimum cardinality constraint is also a particular requirement in this work that is not present in the aforementioned studies.}

The model {\PrimerModelo} considers four sets of binary and continuous variables. Let $y_{i}^c$ be a binary variable taking the value of one if the node $i\in V$ belongs to $V_c$, for some $c \in [k]$, and zero otherwise. The edge-associated variable \(x_{ij}^c\) takes the value of one if the edge \( \{i,j\} \in E\) links nodes \(i\) and \(j\) in the subset \(V_c\), and zero otherwise. A binary variable { \( \bar{x}_{ij}^c \) is defined for each \( \{i,j\}  \in \bar{E}\) } taking the value of one if nodes \(i\) and \(j\) belong to the same cluster \(c\), and zero otherwise. Finally, let \(f_{ij}^{lc}\) be a continuous variable representing the flow over the arc \( (i,j) \) with a source at the node \(l\) on the induced subgraph \( (V_c, E(V_c)) \). Here each edge \(\{i,j\}\) has two associated antiparallel arcs, namely \((i,j)\) and \((j,i)\). Sets  $\delta^-_{j}$ and $\delta^+_{j}$ contain incoming and outgoing arcs to the node $j$, respectively. Thus, the following formulation, denoted as \PrimerModelo, is proposed:
\begin{subequations}
\begin{alignat}{4}
& \min \sum_{c\in [k]} \sum_{\{i,j\}\in E} d_{ij} x_{ij}^c    \label{objective_1.1} \\
\intertext{subject to}
& \sum_{c\in[k]}y_i^c=1 					&&\forall i\in V,  \label{restr_1.2}\\
& y_i^c + y_j^c - x_{ij}^c \leq 1				&&\forall \{i,j\} \in E,  c\in[k], \label{restr_1.3}\\
& y_i^c + y_j^c - 2x_{ij}^c \geq 0			&&\forall \{i,j\} \in E,  c\in[k], \label{restr_1.4}\\
& y_i^c + y_j^c - \bar{x}_{ij}^c \leq 1			&&{\forall \{i,j\} \in \bar{E}},  c\in[k], \label{restr_1.5}\\
& y_i^c + y_j^c - 2 \bar{x}_{ij}^c \geq 0		&&{\forall \{i,j\} \in \bar{E}},  c\in[k], \label{restr_1.6}\\
& \sum_{l\in V} \big( f_{ij}^{lc}+f_{ji}^{lc} \big) \leq (2nk) \, x_{ij}^c 	&&\forall \{i,j\} \in E, c\in[k], \label{restr_1.7}\\
& \sum_{\mathclap{ (i,j) \in \delta^-_{j} } } \, f_{ij}^{lc} \, - \,  \sum_{\mathclap{ ( j,i)\in \delta^+_{j} }} \, f_{ji}^{lc} =
\begin{cases}
	\,\,\, \displaystyle\sum_{\mathclap{ \{j,i\}\not \in E }} - \bar{x}_{ji}^c    &\hspace{-0.25em}\text{ if } j=l, \\
  \,\, 0 												&\hspace{-0.25em}\text{ if }  \{l,j\} \in E, \\
\, \bar{x}_{lj}^c  										&\hspace{-0.25em}\text{ otherwise}, 
\end{cases}
									\hspace{2cm} &&\forall (l,j)\in V\times V, c\in[\p] ,   \label{restr_1.8}		\\
& \sum_{i\in V} y_i^c\geq \alpha  			&&\forall c\in[\p] \label{restr_1.9}				\\
& y_{i}^c \in \{0,1\}  					&&\forall i \in V, c\in [\p], \label{variables_y}			\\
&  x_{ij}^c \in \{0,1\}  					&&\forall \{i,j\} \in E, c \in [k] , \label{variables_x}		\\
&  \bar{x}_{ij}^c \in \{0,1\}  				&&{\forall \{i,j\} \in \bar{E}}, c \in [k] , \label{variables_f}	\\
& f_{ij}^{lc} \in \R_{\geq 0}				&&\forall \{i,j\}  \in E, l\in V, c \in [k] .\label{variables_z}
\end{alignat}
\end{subequations}
The objective function \eqref{objective_1.1} seeks to minimize  the total edge cost of the connected subgraphs $(V_c, E(V_c))$, for all $c\in[k]$. Constraints \eqref{restr_1.2} indicate that each node must belong exactly to one cluster. Constraints \eqref{restr_1.3} and \eqref{restr_1.4} establish that if two nodes $i,j\in V$ are assigned to $V_c$, then the edge $\{i,j\}\in E$ belongs to the induced subgraph $(V_c,E(V_c))$. Similarly, constraints \eqref{restr_1.5} and \eqref{restr_1.6} impose a connectivity relation between non-adjacent nodes within a connected component. Constraints \eqref{restr_1.7} are coupling constraints defining a capacity for the antiparallel arcs associated with edge \(\{i,j\}\in E\). Flow conservation is ensured by constraints \eqref{restr_1.8} for all the nodes within the same connected component. It imposes that one unit of flow is sent through a path linking two nodes $l,j\in V$ in the same connected component whenever $\{l,j\} \in \bar{E}$. 

Let us expand on this explanation by considering each node \(j \in V_c\). Here, \(j\) becomes a source with { \( \sum_{ { \{j,i\} \in \bar{E} }} - \bar{x}_{ji}^c\)} units of flow, which represent the number of nodes in the connected component \(V_c\) that do not share an edge with \(j\). If any other node \( l\in V_c\) is considered such that \(\{l,j\}\) is an edge in \(G\), then the two nodes are connected and no flow is lost. If instead { \( \{l,j\} \in \bar{E}\)} , then a unit of flow is used to signal that a path exists in {\(E(V_c)\)} between \(l\) and \(j\). Overall, these three conditions ensure that each pair of nodes under the same component \(c\) is connected, see also Figure \ref{fig:M1}.

Finally, \eqref{restr_1.9} specify lower bounds for the number of nodes present at each connected component.

\begin{figure}[htbp]
\begin{center}
	\includegraphics[scale=0.8, page=3]{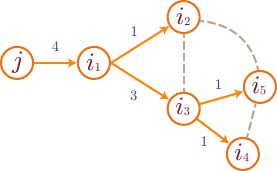}
\caption{Example of a feasible solution for a connected component with \(6\) nodes. Here \(j\) acts as a source from which \(4\) units of flow are distributed to the sink nodes \( i_2\), \(i_3\), \(i_4\), and \(i_5\). There is a direct edge from \(j\) to \(i_1\), hence \(i_1\) is not a sink. The number over each arc represents the flow passing through it. The generated subgraph contains also the edges \( \{i_2,i_3\}\), \(\{i_2,i_5\}\), and \(\{i_4,i_5\}\), however, these do not transfer any flow.}
\label{fig:M1}
\end{center}
\vspace{-1\baselineskip}
\end{figure}

A second formulation is presented using the $k$-augmenting graph of $G$, which is the pair $G^k=(V^k, E^k)$, where $V^k=V\cup \mathbb{A}$, $\mathbb{A}=\{n+1,n+2, \ldots, n+k\}$ is a set of artificial nodes (one for each connected component in the partition), $E^k=E \cup \big\{ \{i,j\}: i \in V, j \in \mathbb{A} \big\}$. For each edge $\{i,j\}\in E$, the original cost $d_{ij}$ is retained, while the edges in $E^k \setminus E$ have a cost set to $0$. This modelling perspective allows a new formulation using three sets of variables. Let \(y_i^c\) be the same as for model \PrimerModelo. In contrast, let \( x_{ij}\) be a binary variable associated with every edge \(\{i,j\} \in E^k\). If $x_{ij}=1$ for \(\{i,j\} \in E\), then edge \(\{i,j\}\) belongs to any connected component. Additionally, let \( f_{ij}\) be a continuous variable representing the flow over the arc \((i,j)\) associated with the edge \(\{i,j\} \in E^k\).  Notice again that each edge has two associated antiparallel arcs.
The second formulation, labelled {\SegundoModelo}, can be stated as follows: 
\begin{subequations}
\begin{alignat}{4}
& \min  \sum_{\{i,j\}\in E} d_{ij} x_{ij}   \label{objective_2.1} \\
\intertext{subject to}
& \sum_{c\in[k]}y_i^c = 1 					&&\forall i\in V,  			\label{restr_2.2}\\
& y_i^c+y_j^c - x_{ij} \leq 1				&&\forall \{i,j\} \in E,  c\in[k], \label{restr_2.3}\\
& y_i^c+y_j^{\ell} + x_{ij} \leq 2			\hspace{3cm}&&\forall \{i,j\} \in E,  c,\ell \in[k], c\neq \ell, 
														\label{restr_2.4}\\
& \sum_{j\in V} x_{ij} = 1					&&\forall  i \in \mathbb{A}, 	\label{restr_2.5}\\
& \sum_{j\in V} f_{ij}=\sum_{j\in V} y_{j}^{i-n}	&&\forall  i \in \mathbb{A}, 	\label{restr_2.6}\\
& \sum_{ \mathclap{(i,j) \in \delta^-_{j} } } f_{ij} - \sum_{\mathclap{ (j,i) \in \delta^+_{j} } } f_{ji}  = 1
									&& \forall j \in V, 			\label{restr_2.7}\\
& f_{ij} +  f_{ji} \leq  \beta x_{ij}				&& \forall \{i,j\} \in E, 		\label{restr_2.8}\\
& \alpha x_{ij}\leq f_{ij} \leq \beta x_{ij} 		&&\forall  i\in \mathbb{A}, j\in V, \label{restr_2.9}\\
& y_{i}^c \in \{0,1\} 						&&\forall i \in V, c\in [\p], 	\label{variables_2.10} \\
&  x_{ij} \in \{0,1\}  						&&\forall \{i,j\} \in E^k,	\label{variables_2.11} \\
&  f_{ij} \in \R_{\geq 0}  					&&\forall \{i,j\}  \in E^k.	\label{variables_2.12}	
\end{alignat}
\end{subequations}
The objective function \eqref{objective_2.1} seeks to minimize  the total edge cost of the connected subgraphs $(V_c, E(V_c))$, for all $c\in[k]$. Constraints \eqref{restr_2.2} indicate that each node must belong exactly to one connected component. Constraints \eqref{restr_2.3}  establish that if two nodes $i,j\in V$ are assigned to $V_c$, then the edge $\{i,j\}\in E$ belongs  to the induced subgraph $(V_c,E(V_c))$, and constraints \eqref{restr_2.4} ensure that the edges with end nodes in different connected components must be equal to zero.  Constraints \eqref{restr_2.5}--\eqref{restr_2.7} impose flow condition to ensure connectivity in the component. Particularly, \eqref{restr_2.6} and \eqref{restr_2.7} are flow conservation constraints.
To see this in more detail, consider the node \(n+c \in \mathbb{A}\) which corresponds to the connected component \(c \in [k]\). Then \eqref{restr_2.6} collects all the nodes in \(c\) and determines this value as the flow that is distributed from the source \(n+c\). 
This is complemented with \eqref{restr_2.7}, where each node \(j \in V_c\) becomes a sink, and \eqref{restr_2.5} that ensures the connectivity of the component by guaranteeing at least one path between each pair of nodes in the connected component. An example is presented in Figure \ref{fig:M2}.

Finally, \eqref{restr_2.8} are coupling constraints imposing capacities for the antiparallel arcs associated with edge \(\{i,j\}\in E\), while \eqref{restr_2.9} define upper and lower bounds on the number of nodes at each connected component. 

\begin{figure}[htbp]
\begin{center}
	\includegraphics[scale=0.8, page=2]{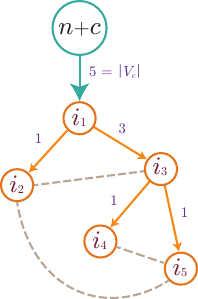}
\caption{Example of a feasible solution for the \(c\)-th connected component with \(5\) nodes. Here the artificial node \(n+c\) acts a source from which \(5\) units of flow are distributed to the sink nodes \( (i_s)_{s \in [5]} \). Flow transfer from \(n+c\) to the nodes in \(V_c\) occurs only once through the arc \( (n+c, i_1)\). The generated subgraph contains the edges \( \{i_2,i_3\}\), \(\{i_2,i_5\}\), and \(\{i_4,i_5\}\), however, these do not transfer any flow.}
\label{fig:M2}
\end{center}
\vspace{-1\baselineskip}
\end{figure}

\subsection{Assisted Column Generation by Spectral Clustering}

In this subsection, the graph partitioning problem in connected components with minimum size constraints  is  modeled as a set partitioning problem, where each node must be assigned to exactly one connected component. Thus, the problem can be modeled as follows: given an undirected graph $G=(V,E)$ with a  cost function $d:E \to \R^+$, and the finite family $\mathcal{F}$ of all induced connected subgraphs of $G$ with number of nodes between \(\alpha\) and \(\beta\), the task consists of finding a subfamily $F\subseteq \mathcal{F}$, with cardinality \(|F|=k\), of pairwise disjoint subgraphs such that
\(
	\bigcup_{f \in F} f = V,
\)
and the total cost is minimized. This formulation resembles the approach taken in Borndörfer et al. \cite{Borndrfer2022} where a different objective is optimized.
 
For this Column Generation approach, the following parameters and variables are needed. Define $a_{if} \in \{0,1\}$ taking the value of one if the connected component $f\in \mathcal{F}$ contains the node $i\in V$, and zero otherwise. Furthermore, let $c_f \geq 0$ be the cost of $f\in \mathcal{F}$, which is equal to the total cost of the edges in the connected component. Now, let \(x_f\) be the decision variable taking the value of one if the connected component \(f\) belongs to a solution, and zero otherwise. Thus, the graph partitioning problem in connected components with minimum size constraints can be formulated as follows {\TercerModelo}:
\begin{subequations}
\begin{align}
	&	\min \sum_{ f\in \mathcal{F}} c_f x_f			\label{obj_2}	\\
\intertext{subject to}
	&  \sum_{f\in \mathcal{F}} a_{if} x_f = 1		& \forall i\in V,\label{restr2_1}\\
	& \sum_{f\in \mathcal{F}} x_f = k,				\label{restr2_2} \\
	& x_f\in \{0,1\}							& \forall f\in \mathcal{F}. \label{rest:3-d}
\end{align}
\end{subequations}
The objective function \eqref{obj_2} seeks to minimize  the total edge cost of the partition. The equality constraints \eqref{restr2_1} indicate that every node must belong to exactly one of the chosen connected components, and constraint \eqref{restr2_2} imposes that exactly $k$ connected components are selected.

The linear relaxation of {\TercerModelo} is denoted by {\TercerModeloRelajado}. Observe that, since only nonnegative edge costs are considered, the upper bound \(x_f \leq 1\) given by the linear relaxation is automatically satisfied.

As the set \(\mathcal{F}\) has an exponential number of elements, a Column Generation approach to solve {\TercerModelo} is justified. In this setting, the model starts with a small subset of connected components $R \subset \mathcal{F}$, and additional components are included throughout an iterative scheme. The relaxed model {\TercerModeloRelajado} starting with the subset of feasible connected components $R$ has an associated dual as follows:
\begin{align*}
	\max \sum_{ i\in V} \pi_i +\p \gamma
	\qquad\text{subject to}\qquad
	\sum_{i\in V} a_{if}\pi_i +\gamma \leq c_f 	\qquad \forall f\in R.
\end{align*}
\noindent where the dual values $\pi_i$ and $\gamma$  are used to decide if the set $R$ must  be expanded or not. 
This yields a Column Generation algorithm, where a connected component \(f\) is added to \(R\) if it has an associated negative reduced cost, i.e., $c_f -  \gamma - \sum_{i\in V} a_{if}\pi_i <0$. Here, each \(f\) is found as a solution of a pricing problem, presented below. The algorithm stops when no additional negative reduced cost connected components are found.
Notice that the Column Generation technique is conventionally applied to the relaxation of the master problem, which becomes a linear program. In this setting, an optimal solution of the relaxation is found as the algorithm stops after no more columns can be added. Hence, the technique can be regarded as a heuristic algorithm for the integral model and does not guarantee an optimal solution. In practice, though, as it will be showcased in the numerical experiments, applying such algorithm does find good quality solutions in competitive times.

The pricing problem can be stated as a minimum node-edge-weighted connected component problem with a minimum size constraint: given an undirected graph $G=(V,E)$ with node weights $-\pi_i$, edge costs $d_{ij}$, and a constant $-\gamma$, find a connected component  whose node set cardinality is bounded between $\alpha$ and $\beta$ and whose total weight is minimized.  
This problem can be formulated as a simplified version of {\PrimerModelo} or {\SegundoModelo}. The simplest approach is given as follows:
\begin{subequations}
\begin{alignat}{3}
	&\min  \sum_{\{i,j\}\in E} d_{ij} x_{ij} - \sum_{i\in V} \pi_{i}y_{i}  - \gamma  
	\label{obj:pricing}
	\intertext{subject to \vspace{-0.5\baselineskip} }
	& y_i+y_j-x_{ij} \leq 1								&&\forall \{i,j\} \in E, \label{restr_p21}	\\
	& \sum_{j\in V} x_{n+1, j} =1,							\label{restr_p22}	\\
	& \sum_{j\in V} f_{n+1, j} = \sum_{j\in V} y_{j}, 				\label{restr_p23}	\\
	& \sum_{ \mathclap{(i,j) \in \delta^-_{j} } } f_{ij} - \sum_{\mathclap{ (j,i) \in \delta^+_{j} } } f_{ji}  = y_{j} 		
											&& \forall j \in V, \label{restr_p23}		\\
	& f_{ij}+f_{ji} \leq \beta x_{ij} 					&& \forall \{i,j\} \in E, \label{restr_p24}	\\
	& \alpha x_{n+1,j} \leq f_{n+1, j} \leq \beta x_{n+1,j} 			\hspace{3cm}
	 										&&\forall j \in V, \label{restr_p25}		\\
	& y_{i} \in \{0,1\}  								&&\forall i \in V, \label{variables_2py}	\\
	&  x_{ij} \in \{0,1\}  								
		&&\forall \{i,j\} \in E \cup \big\{ \{n+1,i\}:\, i\in V \big\}, \label{variables_2px}	\\
	&  f_{ij} \in \R_{\geq 0}  								
		&&\forall \{i,j\}  \in E \cup \big\{ \{n+1,i\}:\, i\in V \big\}. 
		\label{variables_2pf}
\end{alignat}
\end{subequations}
Here, the families of variables \(x\) and \(f\) have the same meaning as in {\SegundoModelo}. However, the binary variables \(y \in \{0,1\}^V\) identify the nodes of only one connected component (in contrast with models {\PrimerModelo} and {\SegundoModelo}). Constraints \eqref{restr_p21}--\eqref{restr_p24} function exactly as \eqref{restr_2.2}--\eqref{restr_2.9} in {\SegundoModelo} for the specialized case \(k = 1\) and \( \mathbb{A} = \{n+1\}\). The only particular difference occurs in constraints \eqref{restr_p23} which enforce that flow can only go through the connected component labeled by \(y\). In a similar fashion, the objective function \eqref{obj:pricing} is just the reduced cost associated with adding the connected component 
to the master problem {\TercerModelo}.

The computational complexity of the pricing problem is proved in the following theorem.

\begin{theorem}
\label{theo:01}
	The pricing problem is {\NP}--hard.
\end{theorem}
\begin{proof}
	This result follows using a polynomial reduction from the Steiner Tree Problem which is well-known to be {\NP}--hard. Let $G=(V,E)$ be a graph with nonnegative edge weights $w_{ij}$ for each edge $\{i,j\}\in E$, and let \(S\) be a subset of $V$. The Steiner Tree Problem \review{consists of} finding a tree in $G$ whose vertices contain $S$ such that the total weight of the tree is minimum. An instance of the pricing problem is constructed as follows: let $\tilde{G}=(\tilde{V}, \tilde{E})$ be an undirected graph such that $\tilde{V}$ contains the node set $V$ alongside one additional node $ij$ for each edge $\{i,j\}$ in $E$.  The set $\tilde{E}$ contains edges of the form $\{i,ij\}$ and $\{ij,j\}$ for any edge $\{i,j\}\in E$. The node weights in $\tilde{G}$ are defined in the following fashion: $ \pi_i = -M $, for each $i\in S$, with $M$ a constant value greater than the total sum of the edge weights of $G$,  $\pi_i=0$  for every $i\in V\setminus S$, and $\pi_{ij}=w_{ij}$ for each $ij\in \tilde{V}\setminus V$.  Moreover, edge weights have values equal to zero for any edge in $\tilde{E}$, and $\gamma$ is fixed to zero. Finally, by construction, every feasible solution $T= \big( V(T), E(T) \big)$ of the Steiner Tree Problem with cost $C(T)=\sum_{\{i,j\}\in E(T)}w_{ij}$ corresponds to a feasible solution of the pricing problem with cost $-M|S|+C(T) +\gamma$   as depicted in Figure \ref{fig:01}.
	\qed
\end{proof}

\begin{figure}[htbp]
\vspace{-0.5\baselineskip}
\begin{center}
	\includegraphics[scale=0.8,page=1]{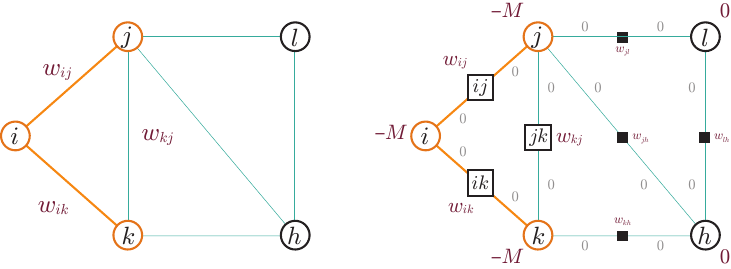}
\caption{Polynomial reduction of the pricing problem to the Steiner Tree Problem. On the left, an instance of the Steiner Tree Problem is presented with \(S=\{i,j,k\}\). On the right, the resulting instance of the pricing problem is presented.}
\label{fig:01}
\end{center}
\vspace{-1\baselineskip}
\end{figure}

A drawback of Theorem {\ref{theo:01}} is that finding the optimal solution of the pricing problem is computationally expensive. Hence, heuristic methods to find columns with negative reduced cost are considered. 
In what follows, a heuristic for finding feasible \(k\) partitions assisted by a spectral clustering algorithm is presented.

Algorithm \ref{alg_Pricing1}  describes a methodology to extend the set $R$. This algorithm takes as input an instance of the Pricing Problem and returns a pool $P$ of connected components with negative reduced cost.  For every node $i$, the algorithm builds several connected components containing such a node. Each one of these, denoted by $C$, is constructed by adding iteratively nodes $\hat{s}$ from $N(C)$ such that the connected component induced by $C\cup \{\hat{s}\}$ has minimum negative reduced cost. If  $|C| \geq \alpha$  and the reduced cost of the generated connected component is negative, then this component is added to the pool $P$.

\begin{algorithm}[!htb]
	\begin{algorithmic}[1]
		\STATE \textbf{Input:} Undirected graph \(G = (V,E)\), shadow price vector \(\pi\), shadow price of column \(\gamma\), size constraints of the connected component \((\alpha,\beta)\), a threshold integer \(\delta \in \mathbb{N}\), {and \( d: E \to \R^{+}\)  a cost function.}
		\STATE Initialize $P = \emptyset$.
		\FOR{ $i \in V$}
		\STATE \(C = \{i\}\).
			\WHILE{$|C| \leq \beta $}
				\STATE Choose node
				\( \displaystyle
					\hat{s} = \argmin_{ s\in N(C) }\bigg\{ \hspace{1.5em}  \sum_{ \mathclap{e \in E[ C \cup \{s\} ] } } d_{e} \hspace{1em} - \quad \sum_{ \mathclap{ i \in C \cup \{s\} } } \pi_i  \bigg\}
				\)
				\STATE $C = C \cup \{\hat{s}\}$
				\STATE 
				\( \displaystyle
					r = \sum_{ \mathclap{e \in E[ C] } } d_{e} - \gamma - \sum_{ \mathclap{ i \in C } } \pi_i
				\)
				\IF{$|C| \geq \alpha$ and \(r < 0 \)}
					\STATE \(P = P \cup \{C\}\)
				\ENDIF
			\ENDWHILE
		\ENDFOR
		\RETURN $\delta$ columns from the column pool $P$ with least reduced cost value.
	\end{algorithmic}
	\caption{Heuristic Pricing Algorithm}
	\label{alg_Pricing1}
\end{algorithm}


Observe that Algorithm \ref{alg_Pricing1} only finds a pool of connected components \(P\) at a time, i.e., columns in \( P\) might not generate a partition of \(G\) into \(k\) connected components until later iterations of the algorithm add enough columns to \(R\).
As a result, the column generation algorithm might have a damped performance, as a full partition \( \{C, V_2, \ldots, V_k\} \) is needed to have a feasible set of connected components. A spectral-clustering algorithm is then presented to find the \emph{complementary partition}, \( \hat V = \{V_2, \ldots, V_k\}\), such that \( \{C\} \cup \hat V\) is a partition of \(V\). {Spectral clustering encompasses a range of techniques aimed at identifying $k$ clusters by using the eigenvectors of a matrix. Usually, this matrix is constructed based on a set of pairwise similarities \( S_{ij} \) among the data points slated for clustering. This task is often referred to as similarity-based clustering, graph clustering, or the clustering of dyadic data. For this purpose, Algorithm \ref{alg_SpectralPartitioning} is presented, which allows to cluster any graph in connected components with a given size.  
\review{Spectral clustering is a remarkably robust and efficient linkage method in statistical analysis, which motivates its integration into combinatorial optimization tasks involving graph partitioning.}
For an extensive explanation of Algorithm \ref{alg_SpectralPartitioning} \review{and its statistical connections}, see \cite{Meila2015}. }
\begin{algorithm}[!htb]
	\begin{algorithmic}[1]
		\STATE \textbf{Input:} Undirected graph \(G = (V,E)\), cost function \(d: E \to \R^+\), and number of clusters \(k\).
		\STATE {Construct a similarity matrix \( S \coloneqq (s_{ij}) \) where \(s_{ij} = d_{ij}\) whenever \( ij \in E\) and \(0\) otherwise.}
	
		\STATE Compute the vector of node degrees \( \hat{d}_i \coloneqq \sum_{j\in V}{d_{ij}}\).
		\STATE Form the transition matrix \( T \) such that \( T_{ij} \coloneqq d_{ij}/\hat{d}_i\).
		\STATE Compute \(k\) largest eigenvectors \(v_{1}, \ldots, v_k\) of \(T\).
		\STATE Embed each node of \(G\) in the \(k-1\) principal subspace; i.e., associate the vector \(x_i \coloneqq [ v_{i,2} \, \ldots \, v_{i,k} ]\) for each node \(i\in V\).
		\STATE Run the constrained \(k\)--means algorithm on the data set \(\{x_i\}_{i \in V}\) to obtain $\hat V$.
		\RETURN $\hat V$.
	\end{algorithmic}
	\caption{Spectral clustering with size constraints}
	\label{alg_SpectralPartitioning}
\end{algorithm}

Particularly, the constrained \(k\)--means clustering problem \cite{bennett2000} arising in Step 7 of Algorithm \ref{alg_SpectralPartitioning} can be solved efficiently  using any linear programming solver. 
The spectral decomposition {in Step 5} is the most expensive operation in Algorithm \ref{alg_SpectralPartitioning}. To alleviate the computational cost of obtaining the spectral information, the Implicit Restarted Arnoldi method can be used in \(\mathcal{O}(nk^2)\) time. 
As a result, Algorithms \ref{alg_Pricing1} and \ref{alg_SpectralPartitioning} can be combined to find a spectral partitioning of \(\hat G = G[V\setminus C]\) that complements any connected component \(C\) in \(G\).

\begin{algorithm}[!htb]
	\begin{algorithmic}[1]
		\STATE \textbf{Input:} Undirected connected graph \(G = (V,E)\), a fixed connected component \(C\), and number of clusters \(k\).
		
		\STATE Compute \(\hat k\), the number of connected components in the induced subgraph \( \hat G \coloneqq G[ V\setminus C ]\). Let \( \{V_i\}_{i \in [\hat{k}]} \) be the family of connected components in \( \hat G\).
	
		\IF{\(\hat{k} > k-1\)}
		\RETURN \(\{\emptyset, C\}\)
        \ENDIF
  
        \IF{\(\hat{k} =k-1\) and \(|V_i|\geq \alpha,~ \forall i \in [\hat k] \)}
		\RETURN \(\{V_1,\dots,V_{k-1}, C\}\)
        \ELSE 
        \RETURN \(\{\emptyset, C\}\)
        \ENDIF

        \IF{\(\hat{k}<k-1\) and \(|V_i|\geq \alpha,~ \forall i \in [\hat k] \)}
        
		\STATE For all \(i \in [\hat k]\), compute \(\kappa_i \coloneqq \left\lfloor \frac{|V_i|}{\alpha} \right\rfloor \), this is the maximum feasible number of subpartitions in which \( V_i \) can be divided. 
				
		\IF{ \(\sum_{i \in [\hat k]}  \kappa_i \geq k-1\) }
		
			\STATE \label{Step12} Obtain a nonempty complementary partition \( \hat V  \) using Algorithm \ref{alg_SpectralPartitioning}~ for each one of the \(\hat k\) connected components. \label{step10}
			\RETURN $\hat V \cup \{C\}$
		
		\ELSE
			\RETURN \(\{\emptyset, C\}\)
			
		\ENDIF
	\ELSE 
        \RETURN \(\{\emptyset, C\}\)
		
		\ENDIF
	\end{algorithmic}
	\caption{\(k\)-feasible partitioning of $G=(V,E)$ from a connected component \(C\)}
	\label{alg_PartitionG}
\end{algorithm}

Algorithm \ref{alg_PartitionG} computes, based on a fixed component \(C\), a feasible $k$ partition from graph $G$. The number \(\hat{k}\) represents the number of connected components in \(\hat G = G[V\setminus C]\) which can be computed using any graph exploration algorithm. Notice that not all possible values of \(\hat k\) give place to a feasible \(k\) partition of the original graph \(G\). As a result, Algorithm \ref{alg_PartitionG} analyses the feasibility of \(\hat k\) and partitionability of the subfamily of connected components \( \{V_i\}_{i \in [\hat{k}]} \) in \(\hat G\). In case that \(\hat G\) cannot be partitioned into \(k-1\) connected components, then Algorithm \ref{alg_PartitionG} returns the pair \( \{\emptyset, C\} \), which certifies that \(C\) cannot belong to any feasible solution of the partitioning problem in connected components with minimum size constraints. Furthermore, in Step \ref{step10} and for very \( i \in [\hat k]\), the induced subgraph \( G[V_i]\), its corresponding similarity matrix, and a cluster size \( p_i \) are given as input for  Algorithm \ref{alg_SpectralPartitioning}. The cluster size \(p_i\) is an integer between \(1\) and \(\kappa_i\) that is computed using a greedy algorithm such that \( \sum_{i\in [\hat k]} p_i = k-1\) and the quotient \( \sum_{e \in E[V_i]} d_e \big / p_i \) is as small as possible. The first condition ensures that the algorithm will return a feasible \(k\) partition, while the second condition reduces the total cost of the subpartitions.

An alternative approach to obtain a nonempty complementary partition \(\hat{V}\) given a connected component \(C\) is based on Constrained Clustering by Spectral Kernel Learning \cite{Li2009}. The main idea of the method is to design a kernel to respect both the proximity structure of the data and the given pairwise constraints. The authors propose a spectral kernel learning framework and formulate it as a convex quadratic program, which can be solved efficiently to optimality. In this context, Step \ref{Step12} of Algorithm \ref{alg_PartitionG} can be executed using Algorithm 1 of  \cite{Li2009} (Spectral Kernel Learning Algorithm) instead of Algorithm \ref{alg_SpectralPartitioning}.

\section{Valid Inequalities}
\label{sec:Valid_ineqs}

\review{Defining valid inequalities is a standard technique in Integer Programming to obtain improved formulations. Specifically, these inequalities are designed to cut-off, or tighten, a fractional part of the feasible region of a linear program and retain integral feasible solutions. 
Additionally, valid inequalities can reduce solution times by decreasing the number of nodes explored in the Branch and Bound (B\&B) tree. 
In this section, various families of valid inequalities are presented and analyzed. A computational study of the effectiveness of each family of inequalities, with respect to the corresponding formulation of the graph partitioning problem, is presented in Section \ref{resultados}.}

The first family of valid inequalities arises if there are leaves in the graph. Thus, if $\ell \in V$ is a node such that $u \in V$ is its unique neighbor, then the edge $\{\ell,u\}$ is part of any feasible solution. Hence, the following inequality holds:

\begin{theorem}\label{th:LC}
	Let $i \in V$ be a node with $|\delta(i)| = 1$ and $N(i) = \{j\}$. The equations
	\begin{align*}
		\sum_{c\in [k]} \hspace{-0.3em} x_{ij}^c=1
		 \qquad\text{and}\qquad 
		\sum_{c\in [k]}  \sum_{l \in V} f_{ji}^{l c}=0
	\end{align*}
	are valid for \PrimerModelo. Additionally,
	\( x_{ij}=1\), \( f_{ji}=1\), and \(f_{ij}=0\)
	are valid for \SegundoModelo.
\end{theorem}
\begin{proof}
	Since $i \in V$ is a leaf in $G$ and $\alpha > 1$, then $\{j,i\}$ belongs to some connected component. For {\PrimerModelo}, at most $\beta - 2$ units of flow are sent to node $i$, meanwhile by flow conservation, zero units of flow are sent in the opposite direction. On the other hand, for {\SegundoModelo}, one unit of flow is sent from node $j$ to  node $i$, and by flow conservation no unit of flow are sent in the opposite direction. 
	\qed
\end{proof}

As a minimum number of nodes is needed for each connected component, there are pairs of nodes that cannot be assigned in the same part.  Recall that all subsets in a partition have at most $\beta$ nodes. Now, let us define \( \DP (i,j)\) as the minimum path length between nodes \(i\) and \(j\).

\begin{theorem}\label{th:SPC}
	Let $i,j\in V$ such that $\DP (i,j)\geq \beta$. Then the inequalities
	\[
		y_i^c + y_j^c\leq 1, \qquad \forall c\in [k]
	\]
	are valid for {\PrimerModelo} and {\SegundoModelo}.
\end{theorem}
\begin{proof}
Since the number of edges in the shortest path from $i$ to $j$,  $\DP (i,j)\geq \beta$, then both nodes cannot belong to the same connected component. 
\qed
\end{proof}

The previous result can be generalized in the following way. {Let \( \LL =\{\ell_1, \ell_2, \ldots\) \(, \ell_q \} \) be a subset of the node set $V$} with $q\leq k$, such that $\DP (\ell_i,\ell_j)\geq \beta$ for all $\ell_i,\ell_j\in \LL$. 

\begin{corollary}
	Observe that if there exists a set $\LL$ with $q>k$, this set becomes an infeasibility certificate for the problem.
\end{corollary}

\begin{theorem}
If $q \in [k]$, the following equations are valid for {\PrimerModelo}: 
\begin{enumerate}[label=(\roman*)]
	\item $y_{\ell_i}^i=1$,  for all $\ell_i\in \LL $, $y_{\ell_i}^j=0,$ $i \neq j $;
	\item $ y_{u}^i=1, y_{u}^j=0,$  for all $ u \in N(\ell_i), \;  v \in \LL\setminus \{\ell_i\} $   such that $\DP(u,v)\geq \beta$,  $i \neq j $;
	\item $ x_{uv}^j=0,$  for all $u \in N(\ell_i) ,  v \in \LL\setminus \{\ell_i\}$   such that $d(u,v)\geq \beta$ and $i \neq j $.
\end{enumerate}
\noindent Additionally, \textbf{(i)} and \textbf{(ii)}  are valid for {\SegundoModelo}.
\end{theorem}

\begin{proof}
	Observe that if $\LL\neq \emptyset$, then $q\geq 2$. Thus, as $\DP(u,v)\geq \beta$, for all $u,v \in \LL$, then every node in $\LL$ belongs to different connected components. In order to avoid symmetric solutions, and without loss of generality, the indexes of nodes in the set $\LL$ can be matched with those in the set of connected components $\{1,2,\ldots,q\}$. Then, one can fix \(y_{\ell_i}^i=1\), for all \(\ell_i\in \LL\) with \(i\in [q]\) and by constraints \eqref{restr_1.2} and \eqref{restr_2.2}, $y_{\ell_i}^j=0$ for all $j\neq i$. On the other hand, if $u\in N(\ell_i)$ is such that $\DP(u,v)\geq \beta$ for all $ v\in \LL\setminus \{\ell_i\}$, then $y_{u}^j=0$ and, by propagation, $x_{uv}^j=0$ for $j\in [q]\setminus \{i\}$.
	\qed
\end{proof}

\begin{corollary}
	If $q=k$ and for each \( i \in [q]\) where there exists $u\in N(\ell_i)$ such that $\DP(u,v)\geq \beta$ with $v\in \LL\setminus \{ \ell_i\}$, then $y_u^i=x_{uv}^i=1$ hold for {\PrimerModelo} and $x_{uv}=1$ holds for  {\SegundoModelo}.
\end{corollary}

\begin{theorem}
	If $q \in [k]$, then the set  $ \mathcal{U}(\ell_i)=\{u\in V: y_u^i=1 \}$ is nonempty for some \(i \in [q]\). If $| \mathcal{U}(\ell_i)| < \alpha$, then the following holds for  {\PrimerModelo} and  {\SegundoModelo}:
	\[
		\sum_{j\in N\big( \mathcal{U}(\ell_i) \big)} \hspace{-1.3em}  y_j^i \geq 1.
	\]
\end{theorem}
\begin{proof}
	As $0 < |\mathcal{U}(\ell_i)| < \alpha$, then at least one node in the neighborhood of \(\mathcal{U}(\ell_i)\) must be included in the $i$-th component to reach the minimum number of nodes.
	\qed
\end{proof}

\begin{theorem}\label{th:S-C}
	Let $S\subset V$ such that $|S| < \alpha$. Then the following is valid for {\PrimerModelo}
	\[
		\sum_{c\in [k]} \sum_{ \{i,j\}\in {\delta(S)}}  \hspace{-0.5em}  x_{ij}^c\geq 1,
	\]
	and the following holds for {\SegundoModelo}
	\[
		\sum_{ \{i,j\}\in {\delta(S)}} \hspace{-0.5em}  x_{ij} \geq 1.
	\]
\end{theorem}

\begin{figure}[htbp]
\vspace{-0.5\baselineskip}
\begin{center}
	\includegraphics[scale=0.8,page=4]{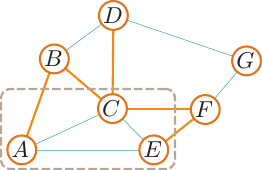}
\caption{\review{An instance of the partitioning problem into connected components with minimum size \(\alpha=4\) is presented to exemplify Theorem \ref{th:S-C}. Here the subset \(S = \{A,C,E\} \subseteq V\) is considered, and its cut is highlighted in thick lines and corresponds to \( \{ AB, BC, CD, CF, EF \} \).
}}
\label{fig:S-C-cut}
\end{center}
\vspace{-1\baselineskip}
\end{figure}

\review{Figure \ref{fig:S-C-cut} depicts the intuition behind Theorem \ref{th:S-C}. Here, an instance with \(7\) nodes and minimum size \( \alpha=4\) is considered. The subset \(S = \{A,C,E\} \subset V \) is highlighted with a dashed rectangle. Due to the minimum size constraint, any feasible solution containing \(S\), must have an additional node belonging to its cut. In other words, Theorem \ref{th:S-C} establishes that the total sum of the variables associated to the edges in $\delta(S)$ must be at least one.}

Furthermore, the following theorems are derived from valid inequalities introduced by Hojny and Miyazawa \cite{Hojny2020,Miyazawa2021}. To state these  kind of valid inequalities, the definition of \emph{articulation node} must be introduced. Here, a node $u \in V$ is an articulation node if the graph obtained by deleting $u$ is disconnected. \review{Figure \ref{fig:Art_node} illustrates the previous definition and its application to Theorem \ref{th:art_node} stated below.}

\begin{theorem}\label{th:art_node}
	Let $\ell \in V$ be an articulation node and let $i$ and $j$ be two nodes from different connected components in the induced graph by $V\setminus \{\ell\}$. The following inequalities are  valid for {\PrimerModelo}  and {\SegundoModelo}:
	\[
		y_i^c + y_j^c-y_{\ell}^c \leq 1, \qquad \forall c\in [ k].
	\]
\end{theorem}


\begin{figure}[htbp]
\vspace{-1\baselineskip}
\begin{center}
	\includegraphics[scale=0.8,page=5]{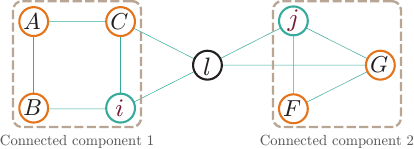}
\caption{\review{Articulation node $\ell$ separating two connected components associated with nodes \(i\) and \(j\).}}
\label{fig:Art_node}
\end{center}
\vspace{-1\baselineskip}
\end{figure}

The following concept is useful to generalize the previous result. Let $u$ and $v$ be two non-adjacent nodes in the graph, the set $S\subseteq V\setminus \{u,v\}$ is said a \((u,v)\)--\emph{separator} if $u$ and $v$ belong to different components in the graph induced by  $V\setminus S$.  Let $\mathcal{S}(u,v)$ be the collection of all minimal $(u, v)-$separators in $G$. Then, next result guarantees a connected subgraph. 

\begin{theorem}\label{th:(u,v)SC}
	Let $i,j\in V$ be two non-adjacent nodes and let $S\in \mathcal{S}(i,j)$ be a \((i,j)\)--\emph{separator}. The following inequalities are  valid for {\PrimerModelo}  and {\SegundoModelo}:
	\[
		y_i^c + y_j^c - \sum_{\ell \in S} y_{\ell}^c \leq 1, \qquad \forall c\in [ k].
	\]
\end{theorem}

The following two valid inequalities have a simple structure and define a simple lower bound for the number of edges in a feasible solution.

\begin{theorem}\label{th:LBC}
	The inequality 
	\[
		\sum_{c\in [k]} \sum_{ \{i,j\}\in E} \hspace{-0.4em}  x_{ij}^c \geq n-k 
	\]
	is valid for {\PrimerModelo}, and
	\[
		\sum_{ \{i,j\}\in E} \hspace{-0.3em}  x_{ij} \geq n-k 
	\]
	is valid for {\SegundoModelo}.
\end{theorem}

For any subset of nodes $C$ with $|C| > k$  such that $(C,E(C))$ forms a clique, then it holds that several edges of the clique lie in the same connected component. Thus, the so-called clique inequalities hold:

\begin{theorem}\label{th:CC}
Let $C\subset V$, such that $(C,E(C))$ is a clique with $|C| > k$. Let $q =|C| -k $, then
\[
	\sum_{c\in [k]} \sum_{ \{i,j\}\in E(C) } \hspace{-0.7em} x_{ij}^c\geq 
	\begin{cases}
		\max\left\{ \left\lfloor \frac{|C|}{2} \right\rfloor \left( \left\lfloor \frac{|C|}{2} \right\rfloor -1\right), q \right\} &\text{ if } k=2,\\ 
	q &\text{ if }  k>2,
\end{cases}
\]
are valid inequalities for {\PrimerModelo}. In similar way, 
\[
	\sum_{ \{i,j\}\in E(C) } \hspace{-1em} x_{ij}\geq 
	\begin{cases}
		\max\left\{ \left\lfloor \frac{|C|}{2} \right\rfloor \left( \left\lfloor \frac{|C|}{2} \right\rfloor -1\right), q \right\} &\text{ if } k=2,\\ 
	q &\text{ if }  k>2,
	\end{cases}
\]
are valid inequalities for {\SegundoModelo}.
\end{theorem}
\begin{proof}
	Let us define the integer \(p = k +q\) as the number of nodes in \(C\), and start with the case \(k =2\) for any other integer \( q\geq 1\).
	Notice that a clique of size \(4\) must have at least \(2\) edges in a \(2\)-connected partitioning, and the cases \(p=3\) and  \(p=5\) are similar. 
	As a result, let us further assume that \(p\geq 6\) is even without losing generality. 
	In this setting, \( \frac{p}{2} (\frac{p}{2} - 1) = \frac{q}{2} (1 + \frac{q}{2})> q\), and this is  the smallest number of edges from \(C\) in a \(2\)-connected partitioning. To see this, notice that this quantity is just the result of evenly splitting \(C\) in two parts, namely \( C = P_1 \cup P_2\), and assigning each half to a different connected component.
	Now, if a node \(\imath\) is taken from \( P_1\) and assigned to the other component, then \( P_2 \cup \{\imath\}\) would form a clique of size \( \frac{p}{2} + 1\), for the number of edges from \(C\) in this partitioning changes to
	\(
		\frac{1}{2} (\frac{p}{2} + 1)\frac{p}{2} 
		+
		\frac{1}{2} (\frac{p}{2} - 1) (\frac{p}{2} - 2)
		=
		\frac{1}{2} (\frac{q}{2} + 2) (\frac{q}{2} +1)
		+
		\frac{1}{2} \frac{q}{2} (\frac{q}{2} - 1),
	\)
	which is greater than \(\frac{q}{2} (1 + \frac{q}{2})\).

	The proof for \(k> 2\) is done by induction over \(|C|\) as follows. 
	The simplest case is when \( |C| = 4\) with \( k = 3\) and \(q = 1\). Here it is clear that at least two nodes from \(C\) must belong to the same connected component, and hence at least \(q=1\) edges from the clique will be included in the solution.
	%
	%

	Let us assume that for a clique \(C\)  with \(p-1\) nodes at least
	\[
		S_{p-1} := 
		\begin{cases}
			\max\Big\{ \big\lfloor \frac{p-1}{2} \big\rfloor \big( \big\lfloor \frac{p-1}{2} \big\rfloor -1\big), q \Big\} &\text{ if } k=2,
			\\ 
			q &\text{ if }  k>2,
		\end{cases}
	\]
	edges from \( (C,E(C))\) are present in a \(k\)-connected partitioning. 
	
	Now, consider the case \( p = k +q\) again with \(k > 2\). Here select a node \( \imath \in C\) and identify its assigned connected component as \( V_\ell \), with some \(\ell \in [k]\). There are two scenarios: either \(\imath\) is the only node from \(C\) in \(V_\ell\), or there exists another \(\jmath \in C\cap V_\ell\) different from \(\imath\).
	In the former case,  \( C\setminus \{\imath\}\) has to be distributed in \(k-1\) components. Notwithstanding, this is a clique of size \( p-1 = (k-1) + q\), where at least \( S_{p-1}\) edges are present in a connected (\(k-1\))-connected partitioning,
	%
	which is greater or equal than \(q\). 
	The latter case is similar. Without losing generality, it can be assumed that \(|C\cap V_\ell| = 2\) (the argument when \(|C\cap V_\ell| > 2\) is similar), then \( |C\setminus \{\imath,\jmath\}| = (k-1) + (q-1) = p-2\). Again \( S_{p-2} \geq q-1\), and as the edge \(\{\imath,\jmath\}\) is in \( \big(V_\ell, E(V_\ell) \big) \), then there are at least \(q\) edges from \(C\) in a \(k\)-connected partitioning. 
	\qed
\end{proof}

For the sake of simplicity, the number of nodes in a connected component \(f\in \mathcal{F}\) is denoted by \(|f|\). Additionally, let \(H(\ell) \) be the subset of connected components in \(\mathcal{F}\) of size \(\ell\geq \alpha\).

\begin{theorem}\label{th:Base-M3-ImplementedCut}
	For every integer \(\ell\neq \alpha\) such that \( \left \lceil \frac{n}{k} \right \rceil \leq \ell \leq \beta\), then the following inequality is valid for {\TercerModelo}:
	\[
		\sum_{ \mathclap{f\in H(\ell)}} x_f \leq 
		\left\lfloor \frac{n-k\alpha}{\ell-\alpha} \right\rfloor .
	\]
\end{theorem}
\begin{proof}
	It is clear that a feasible solution cannot have more than \( \left\lfloor \frac{n}{\ell} \right\rfloor\) columns from \(H(\ell)\).
	However, this bound can be refined in the following way: Let \(\theta\) be the number of connected components from \(H(\ell)\) in a feasible solution. Then the remaining \(k-\theta\) connected components in such solution must have at least \(\alpha\) nodes. As a consequence, \(\theta \ell +(k-\theta) \alpha \leq n\). 
	By taking the maximum value of \(\theta\) satisfying this relation, the result follows.
	 \qed
\end{proof}

The previous theorem can be generalized:

\begin{theorem}\label{th:M3-ImplementedCut}
For every integer \(\ell\) such that \( \left\lceil\frac{n}{k} \right\rceil \leq \ell \leq \beta\),
then the following inequality is valid for {\TercerModelo}:
\[
	\sum_{i=\ell}^{\beta }  \sum_{f\in H(i)} \hspace{-0.5em} x_f
	\leq  
	\left\lfloor \frac{n-k\alpha}{\ell-\alpha} \right\rfloor .
\]
\end{theorem}
{
\begin{proof}
According to the proof of Theorem \ref{th:Base-M3-ImplementedCut} , let $\theta_i$ be the number of connected components of size \(i\), with \(\ell \leq i \leq \beta\). Then:
\begin{align*}
\sum_{i=\ell}^{\beta}i \theta_i +\left(k-\sum_{i=\ell}^{\beta}\theta_i\right)\alpha
\leq n
\qquad \iff \qquad
\sum_{i=\ell}^{\beta}\theta_i (i-\alpha)\leq n-k\alpha.
\end{align*}
Observe that \(\sum\limits_{i=\ell}^{\beta}\theta_i (i-\alpha) \geq \sum\limits_{i=\ell}^{\beta}\theta_i (\ell-\alpha)\), and therefore
\[
\sum_{i=\ell}^{\beta }  \sum_{f\in H(i)} \hspace{-0.5em} x_f = \sum_{i=\ell}^{\beta}\theta_i
	\leq  
	\left\lfloor \frac{n-k\alpha}{\ell-\alpha} \right\rfloor. \qed
\] 
\end{proof}
}

\begin{theorem}
Let  \(\ell\)  be an integer such that  \( \alpha \leq \ell \leq \beta\), \( S\subset V\) a subset with \(|S|\geq \ell\), and \(  \theta =  \lfloor |S| /  \ell \rfloor  \). Then,  the following inequality is  valid for  {\TercerModelo}:
\[
	\sum_{i=\ell}^{\beta} \sum_{\substack{f\in \mathcal{F}(S):\\|f|=i}} \hspace{-0.5em} x_f\leq \min \{ k-1,\theta\}.
\]
Here, $\mathcal{F}(S)\subset \mathcal{F}$ is a subfamily of connected components in  $(S, E(S))$. 
\end{theorem}
\begin{proof}
The result follows from observing that the subset $S$ cannot contain more than $\theta$ connected components with at least $\ell$ nodes.
\qed
 \end{proof}

\section{Computational Experiments}
\label{resultados}

Computational experiments with our MIP and IP formulations and valid inequalities are carried out in this section. Three sets of tests are presented aiming to solve a wide set of instances that test our exact and heuristic approaches for the partitioning problem into connected components. Emphasis \review{was placed on} testing the different valid inequalities introduced in the previous section as well as evaluating the parametric behaviour of the heuristic.
In this regard, a first set of tests  are performed and conducted by solving the problem using exact algorithms associated with formulations {\PrimerModelo}  and {\SegundoModelo} and combining them with valid inequalities in different ways. 
A second set of experiments compares the proposed spectral clustering method against the Constrained Clustering by Spectral Kernel Learning method of \cite{Li2009}. Solving the column generation formulation {\TercerModelo}, along with the two resulting versions of the spectral-assisted heuristic allows us to have a clear comparison of our proposed method against another clustering algorithm already established in the literature.
At last, a third set of experiments uses the column generation formulation {\TercerModelo} along with the assisted heuristic method designed for solving the underlying pricing problem. The formulations were solved using the MIP solver Gurobi 9.5.1 \cite{Gurobi} with its Python 3.7 API. All the experiments were performed on an Intel Core i7 $3.40$ GHz with $10$ GB RAM running Ubuntu $22.04$. The computation time is limited to \MaxText\(=3600\) seconds for every instance. 

The instances were generated considering different simulated graphs, lower bounds $\alpha$, and partition sizes \(k\). 
The number of nodes was taken between 40 and 70 
reflecting model sizes that were tractable for building the exact MIP formulations and compare their performance efficiently.
Diversity has been sought regarding the density of the graphs, ranging from very sparse (density \(\sim 0.08\)) to moderately dense (density \(\sim 0.71\)). 
The choice of \(\alpha \geq 2\) prevents solutions with components consisting of just one isolated node.
Moreover, constraining \(\alpha\) to a lower range enables the examination of a vast feasible region. This is because larger values of \(\alpha\) lead to the exclusion of numerous connected components, indicating that lower values of this parameter correspond to more challenging instances.
Finally, the range \(4\leq k \leq 10\) comes as a common range for clustering applications. 
%

\begin{table}[]
\setlength{\tabcolsep}{0.4em}
\centering
\fontsize{6.5}{7.5}\selectfont
\def\arraystretch{1.5}
\caption{Solving {\PrimerModelo} and {\SegundoModelo} using a Brach \& Cut algorithm.}
\label{Table:M1M2}
\begin{tabular}{cl rrrrr c rrrrr}
\toprule
\hiderowcolors
\multirow{4}{*}{\begin{tabular}[c]{@{}c@{}}\bf Instance\\ $(n, |E|)$ \\ $(\alpha, k)$ \\ Density \end{tabular}} & \multicolumn{1}{c}{\multirow{4}{*}{\bf Cuts}} & \multicolumn{5}{c}{\PrimerModelo} &  & \multicolumn{5}{c}{\textbf{\SegundoModelo}} 
\\[0.2em] \cline{3-7} \cline{9-13} 
 &  & 
 	\multirow{3}{*}{\bf LB} & \multirow{3}{*}{\bf Obj} & \multirow{3}{*}{\bf Gap} & \multirow{3}{*}{ \begin{tabular}[c]{@{}c@{}}\bf B\&B \\[-0.5em] \bf nodes \end{tabular} } & \multirow{3}{*}{\bf Time} &  & 
 \multirow{3}{*}{\bf LB} & \multirow{3}{*}{\bf Obj} & \multirow{3}{*}{\bf Gap} & \multirow{3}{*}{ \begin{tabular}[c]{@{}c@{}}\bf B\&B \\[-0.5em] \bf nodes \end{tabular} } & \multirow{3}{*}{\bf Time}
 \\[-0.2em]
  \\
 \\ \midrule
\multirow{7}{*}{ \begin{tabular}[c]{c} (30,82)$$ \\[0.5em] $(7,4)$ \\ 0.19 \end{tabular} } 
 & \No & 0 & 83 & 0 & 789936 & 3218 &  & 0 & 84 & 45 & 10892614 & \MaxT \\
 & \LC & 7 & 83 & 16 & 630155 & \MaxT &  & 13 & 84 & 45 & 7916565 & \MaxT \\
 & \SPC & 0 & 83 & 22 & 443806 & \MaxT &  & 7 & 83 & 43 & 8463839 & \MaxT \\
 & \LBC & 43 & 83 & 0 & 8052 & 60 &  & 47 & 83 & 0 & 7628644 & 2708 \\
 & \SPC, \LBC & 43 & 83 & 0 & 6017 & 44 &  & 47 & 83 & 0 & 1430379 & 496 \\
 & \LC, \SPC, \LBC & 47 & 83 & 0 & 7333 & 46 &  & 50 & 83 & 0 & 1119449 & 375 \\
 & \SC & 0 & 83 & 23 & 27579 & \MaxT &  & 7 & 84 & 46 & 10284659 & \MaxT \\ \hline
\multirow{7}{*}{ \begin{tabular}[c]{c} (35,60)$$ \\[0.5em] $(7,5)$ \\ 0.10 \end{tabular} } 
 & \No & 0 & 113 & 32 & 238606 & \MaxT &  & 11 & 113 & 27 & 14781290 & \MaxT \\
 & \LC & 11 & 113 & 0 & 334652 & 2960 &  & 14 & 113 & 26 & 14935646 & \MaxT \\
 & \SPC & 0 & 113 & 25 & 333303 & \MaxT &  & 11 & 113 & 27 & 11650980 & \MaxT \\
 & \LBC & 78 & 113 & 0 & 7249 & 117 &  & 79 & 113 & 0 & 32042 & 11 \\
 & \SPC, \LBC & 78 & 113 & 0 & 882 & 14 &  & 79 & 113 & 0 & 32042 & 13 \\
 & \LC, \SPC, \LBC & 79 & 113 & 0 & 1477 & 15 &  & 79 & 113 & 0 & 27586 & 11 \\
 & \SC & 0 & 115 & 26 & 18861 & \MaxT &  & 11 & 118 & 18 & 28229 & \MaxT \\ \hline
\multirow{7}{*}{ \begin{tabular}[c]{c} (40,136)$$ \\[0.5em] $(5,7)$ \\ 0.17 \end{tabular} } 
 & \No & 0 & 83 & 45 & 65550 & \MaxT &  & 17 & 93 & 57 & 1333975 & \MaxT \\
 & \LC & 17 & 84 & 44 & 74197 & \MaxT &  & 21 & 89 & 61 & 1865731 & \MaxT \\
 & \SPC & 0 & 83 & 45 & 56505 & \MaxT &  & 17 & 93 & 57 & 1477220 & \MaxT \\
 & \LBC & 41 & 81 & 0 & 15777 & 532 &  & 51 & 88 & 27 & 1204068 & \MaxT \\
 & \SPC, \LBC & 41 & 81 & 0 & 5104 & 208 &  & 51 & 88 & 27 & 1239213 & \MaxT \\
 & \LC, \SPC, \LBC & 51 & 81 & 0 & 3368 & 171 &  & 52 & 83 & 27 & 1511141 & \MaxT \\
 & \SC & 0 & 81 & 21 & 5746 & \MaxT &  & 17 & --- & --- & 10306 & \MaxT \\	 \hline
\multirow{7}{*}{ \begin{tabular}[c]{c} (45,147)$$ \\[0.5em] $(5,8)$ \\ 0.15 \end{tabular} } 
 & \No & 0 & 91 & 46 & 47726 & \MaxT &  & 18 & 103 & 65 & 1151606 & \MaxT \\
 & \LC & 18 & 92 & 48 & 74901 & \MaxT &  & 28 & 99 & 63 & 1300198 & \MaxT \\
 & \SPC & 0 & 91 & 46 & 43167 & \MaxT &  & 18 & 103 & 66 & 1001731 & \MaxT \\
 & \LBC & 52 & 91 & 10 & 61276 & \MaxT &  & 65 & 94 & 23 & 893219 & \MaxT \\
 & \SPC, \LBC & 52 & 91 & 0 & 18351 & 906 &  & 65 & 94 & 23 & 912087 & \MaxT \\
 & \LC, \SPC, \LBC & 65 & 91 & 12 & 56061 & \MaxT &  & 69 & 98 & 18 & 896814 & \MaxT \\
 & \SC & 0 & 98 & 33 & 3526 & \MaxT &  & 18 & --- & --- & 7487 & \MaxT \\	\hline
\multirow{7}{*}{ \begin{tabular}[c]{c} (55,96)$$ \\[0.5em] $(6,9)$ \\ 0.06 \end{tabular} } 
 & \No & 0 & 175 & 46 & 12725 & \MaxT &  & 22 & 175 & 47 & 792432 & \MaxT \\
 & \LC & 22 & 175 & 46 & 15201 & \MaxT &  & 38 & 222 & 60 & 864840 & \MaxT \\
 & \SPC & 0 & 175 & 46 & 25391 & \MaxT &  & 22 & 197 & 64 & 1628049 & \MaxT \\
 & \LBC & 128 & 175 & 2 & 50531 & \MaxT &  & 134 & --- & --- & 741070 & \MaxT \\
 & \SPC, \LBC & 128 & 175 & 0 & 1157 & 93 &  & 134 & 175 & 5 & 1217239 & \MaxT \\
 & \LC, \SPC, \LBC & 134 & 175 & 0 & 6474 & 1084 &  & 140 & 175 & 5 & 1627349 & \MaxT \\
 & \SC & 0 & --- & --- & 2297 & \MaxT &  & 22 & --- & --- & 6337 & \MaxT \\	\hline
\multirow{7}{*}{ \begin{tabular}[c]{c} (60,70)$$ \\[0.5em] $(6,9)$ \\ 0.04 \end{tabular} } 
 & \No & 0 & 233 & 0 & 15693 & 3461 &  & 80 & 235 & 43 & 1108251 & \MaxT \\
 & \LC & 80 & 233 & 39 & 26539 & \MaxT &  & 113 & --- & --- & 1837900 & \MaxT \\
 & \SPC & 0 & 233 & 41 & 17250 & \MaxT &  & 80 & --- & --- & 1235168 & \MaxT \\
 & \LBC & 196 & 233 & 12 & 7765 & \MaxT &  & 199 & --- & --- & 947137 & \MaxT \\
 & \SPC, \LBC & 196 & 233 & 11 & 12568 & \MaxT &  & 199 & 246 & 11 & 974604 & \MaxT \\
 & \LC, \SPC, \LBC & 199 & 233 & 11 & 37120 & \MaxT &  & 204 & 233 & 1.29 & 3426116 & \MaxT \\
 & \SC & 0 & --- & --- & 2637 & \MaxT &  & 80 & --- & --- & 5877 & \MaxT 
\\ \bottomrule

\end{tabular}
\\[0.5em]
\flushleft
{Lower bound from linear relaxation (LB), objective value (Obj), gap percentage (Gap), number of explored nodes in the Branch \& Bound tree, and time taken for each instance and cut configuration per model. The symbol {\MaxT} is reported whenever the maximum runtime {\MaxText} was reached.}
\vspace{-1\baselineskip}
\end{table}

The Python libraries \texttt{iGraph} and \texttt{NetworkX} were used for graph construction with several values on the number of nodes \(n\) and edges \(|E|\), respectively. 
In particular, \texttt{NetworkX} was used to simulate dense graphs uniformly sampled at random. More details on the simulation can be found in our open code repository.
For the spectral-assisted heuristic, \texttt{SciPy} was used to compute eigen-information, and \texttt{k-means-constrained} was used for finding clusters based on the spectral encoding described in Algorithm \ref{alg_SpectralPartitioning}. For each instance, several experiments were performed by applying different configurations of valid inequalities for models {\PrimerModelo}  and {\SegundoModelo}.  The inequalities considered are Theorem \ref{th:LC} ({\LC}), Theorem \ref{th:SPC} ({\SPC}), Theorem \ref{th:LBC} ({\LBC}),  Theorem \ref{th:S-C} ({\SC}), and Theorem \ref{th:M3-ImplementedCut} ({\CGC}). All valid inequalites \LC, \SPC, and \LBC~ were added as constraints in the root node of the Branch \& Bound algorithm for both formulations. On the other hand, regarding valid inequalities \SC, the following separation routine was used for {\PrimerModelo} and {\SegundoModelo}: for every subset $S\subset V$, with \(2\leq |S|\leq 3\), if the inequality of Theorem \ref{th:S-C} associated with such subset \(S\) is infringed, then the inequality is added to the corresponding linear relaxation. Additionally, \( (\beta - \lceil n/k \rceil)+1\) valid inequalities of type {\CGC} are  included in the root node of formulation {\TercerModelo}.

The first set of experiments is presented in Table \ref{Table:M1M2}.  The first column displays the instances identified by the 4--tuple $(n,|E|,\alpha,k)$ and the graph density (\( \nicefrac{2|E|}{n(n-1)} \)) is also reported.
The second column describes the configurations used by {\PrimerModelo}. The third to seventh columns report the lower bound from the linear relaxation, the objective function value, the optimality gap percentage, the number of B\&B nodes evaluated in the optimization process, and the CPU time in seconds, respectively. The remaining columns show the lower bound, the objective function value, the optimality gap, the number of Branch \& Bound (B\&B) nodes  and the CPU time in seconds for different configurations of {\SegundoModelo}. All IP formulations were solved using a cut-disabled configuration of Gurobi {(namely, we set the following parameter values \texttt{PreCrush = 1}, \texttt{Cuts = 0}, \texttt{Presolve = 0}).}

The experiments show a remarkable behavior of all the implemented valid inequalities. In particular, the {\LBC} inequality decreases significantly the number of explored nodes of the B\&B tree and the optimality gap. However, for a reduced number of instances, {\LBC} works pretty well when it is combined with {\SPC} (see for example instance $(55,6,9)$).  Additionally, it is important to see that all instances for the first formulation are solved up to optimality considering several configurations.  Specifically, {\PrimerModelo} with {\SPC} \& {\LBC}  reports the smallest CPU time in most of these instances.  \review{Note also that certain combinations of valid inequalities are more effective than others as these combinations approximate the convex hull of the set of feasible integer points. Indeed, observe that the combination {\LC}, {\SPC}, and {\LBC} produces a lower bound that is greater than or equal to any other combination of inequalities, regardless of model {\PrimerModelo} or {\SegundoModelo}. 
At last, observe that any combination of inequalities helps the solver to determine a higher lower bound for all instances of {\SegundoModelo}.
}

\begin{table}[!htb]
\setlength{\tabcolsep}{0.5em}
\centering
\fontsize{6.5}{7.5}\selectfont
\def\arraystretch{1.5}
\caption{Comparison of different spectral methods for {\TercerModelo} }
\label{Table:ColGen_Spectrals}
\begin{tabular}{l c ccccc c ccc c ccc}
\toprule
\hiderowcolors
\hspace{0.5em}\multirow{2}{*}{\begin{tabular}[c]{@{}c@{}}\bf Instance\\ $(n, |E|,  k, \alpha)$\end{tabular}} & \multirow{2}{*}{\bf Density}
 & \multicolumn{4}{c}{ {\TercerModelo} \& {\CGC}--A } &  & \multicolumn{4}{c}{ {\TercerModelo} \& {\CGC}--B} 
 \\
 \cline{3-6} \cline{8-11}
 &&  {\bf Obj} & {\bf Nodes} & {\bf Cols} & {\bf IP Time} & & {\bf Obj} & {\bf Nodes} & {\bf Cols} & {\bf IP Time}
 \\ \midrule
(40, 624, 2, 5)  	& 0.80 & 1554  & 1  & 7300  & 0,08    & & 1557  & 1  & 7720  & 0,13 \\
(50, 919, 3, 7)  	& 0.75 & 1311  & 1  & 8922  & 0,11    & & 1327  & 15  & 6084  & 0,11 \\
(60, 1\,062, 3, 8)  	& 0.60 & 1823  & 130  & 6481  & 2,23    & & 1905  & 270  & 6933  & 3,39 \\
(70, 1\,328, 4, 10)  	& 0.55 & 1950  & 3916  & 8496  & 15,05    & & 1973  & 5186  & 9054  & 18,91 \\
(80, 1\,422, 4, 15)  	& 0.45 & 2016  & 3631  & 10882  & 20,03    & & 2048  & 3620  & 11012  & 21,59 \\
(90, 1\,602, 5, 12) 	& 0.40 & 2048  & 45378  & 15497  & 421,84    & & 2094  & 31181  & 15011  & 297,45 \\
(100, 1\,732, 6, 10) 	& 0.35 & 1926  & 656090  & 21383  & 3378,71    & & 2022  & 548905  & 21724  &  {\MaxT}
\\
\bottomrule
\end{tabular}
\\[0.5em]
\footnotetext
\flushleft
{Objective value (Obj), gap percentage (Gap), and time taken for each instance and configuration per model. The symbol {\MaxT} is reported whenever the maximum runtime {\MaxText} was reached. All tests were run with $\delta = 40$.  \hfill
}
\end{table}

{%
The second set of experiments compares our proposed spectral-assisted heuristic, Algorithm \ref{alg_PartitionG}, and its variant by replacing Step \ref{Step12} with Algorithm 1 of \cite{Li2009}. The two variants are denoted by {\CGC}--A and {\CGC}--B, respectively.
The results are displayed in Table \ref{Table:ColGen_Spectrals}. Similarly to the previous test, the first column describes each instance by its size and partition parameters, the graph density is included in the second column, the third to sixth columns and the seventh to last columns summarize the outputs of solving {\TercerModelo} with {\CGC}--A and {\CGC}--B, respectively. 
Precisely, the value of the objective function, the number of explored nodes in the B\&B process, the total number of columns (i.e., connected components) that were added by the selected heuristic, and the CPU time in seconds for solving the IP were reported for each instance.
}

\begin{table}[!htb]
\setlength{\tabcolsep}{0.5em}
\centering
\fontsize{6.5}{7.5}\selectfont
\def\arraystretch{1.5}
\caption{Comparison of {\PrimerModelo} and {\SegundoModelo} with {\TercerModelo} }
\label{Table:ColGen}
\begin{tabular}{l c ccccc c ccc c ccc}
\toprule
\hiderowcolors
\hspace{0.5em}\multirow{2}{*}{\begin{tabular}[c]{@{}c@{}}\bf Instance\\ $(n, |E|,  k)$\end{tabular}} & \multirow{2}{*}{\bf Density}
 & \multicolumn{5}{c}{ {\TercerModelo} \& \CGC } &  & \multicolumn{3}{c}{ { \PrimerModelo} \& \LBC-\SPC} &  & \multicolumn{3}{c}{ {\SegundoModelo} \& \LBC-\SPC} 
 \\
 \cline{3-7} \cline{9-11} \cline{13-15}
 &&  {\bf Obj} & {\bf Gap} & {\bf $\delta$ } & {\bf Iters} & {\bf Time} & &  {\bf Obj} & {\bf Gap} & {\bf Time} &  & {\bf Obj} & {\bf Gap} & {\bf Time}
 \\ \midrule
(40, 312, 7)  & 0.40 & 55  & 0.00  & 60  & 28  & 0,12    & &62  & 43,55 & \MaxT & &62  & 45,17 & \MaxT \\
(40, 390, 7)  & 0.50 & 59  & 0.00  & 100 & 21  & 0,33    & &67  & 50,75 & \MaxT & &70  & 55,72 & \MaxT \\
(40, 468, 7)  & 0.60 & 72  & 0.00  & 60  & 20  & 0,54    & &79  & 58,23 & \MaxT & &85  & 64,71 & \MaxT \\
(40, 507, 7)  & 0.65 & 88  & 0.00  & 80  & 19  & 1,05    & &107 & 69,16 &\MaxT & &95  & 67,37 & \MaxT \\
(40, 550, 7)  & 0.71 & 78  & 0.00  & 70  & 24  & 0,28    & &94  & 64,9  & \MaxT & &105 & 70,48 & \MaxT \\
\hline
(45, 400, 8)  & 0.40 & 59  & 0.00  & 40  & 23  & 0,07    & &120 & 69,17 & \MaxT && 66  & 43,94 & \MaxT \\
(45, 495, 8)  & 0.50 & 65  & 0.00  & 70  & 25  & 1,29    & &99  & 62,63 & \MaxT& &78  & 58,98 & \MaxT \\
(45, 600, 8)  & 0.61 & 74  & 0.00  & 40  & 26  & 0,7     & &85  & 56,48 & \MaxT & &77  & 57,15 & \MaxT \\
(45, 643, 8)  & 0.65 & 91  & 0.00  & 90  & 19  & 1,99    & &110 & 66,37 & \MaxT  & &97  & 69,08 & \MaxT \\
(45, 700, 8)  & 0.71 & 103 & 0.00  & 100 & 18  & 0,97    & &121 & 69,43 & \MaxT & &120 & 74,17 & \MaxT \\
\hline
(50, 105, 8) & 0.09 & 283 & 0.00  & 40  & 100 & 0,04    & &158 & 17,09 & \MaxT & &145 & 20,69 & \MaxT \\
(50, 220, 8) & 0.18 & 102 & 0.00  & 90  & 31  & 0,32    & &153 & 54,25 & \MaxT & &104 & 29,81 & \MaxT \\
(50, 298, 8) & 0.24 & 69  & 0.00  & 40  & 37  & 0,14    & &137 & 62,05 & \MaxT & &72  & 33,34 & \MaxT \\
(50, 408, 8) & 0.33 & 84  & 0.00  & 70  & 28  & 2,64    & &133 & 67,67 & \MaxT & &86  & 46,52 & \MaxT \\
(50, 529, 8) & 0.43 & 79  & 0.00  & 90  & 22  & 1,73    & &132 & 68,19 & \MaxT & &96  & 57,3  & \MaxT \\
\hline
(55, 350, 8)  & 0.24 & 85  & 0.00  & 70  & 41  & 0,45    & &127 & 53,55 & \MaxT & &101 & 42,58 & \MaxT \\
(55, 445, 8)  & 0.30 & 92  & 0.00  & 80  & 43  & 2,55    & &129 & 57,37 & \MaxT & &--  & --    & \MaxT \\
(55, 550, 8)  & 0.37 & 83  & 0.00  & 40  & 35  & 2,56    & &149 & 68,46 & \MaxT & &96  & 54,17 & \MaxT \\
(55, 594, 8)  & 0.40 & 127 & 0.00  & 90  & 28  & 153,62  & &183 & 74,32 & \MaxT && --  & --    & \MaxT \\
(55, 800, 8)  & 0.54 & 110 & 0.00  & 60  & 33  & 6,36    & &182 & 74,18 & \MaxT & &215 & 81,4  & \MaxT \\
(55, 891, 8)  & 0.60 & 137 & 0.00  & 100 & 31  & 24,53   & &229 & 79,04 & \MaxT & &116 & 65,52 & \MaxT \\
\hline
(60, 150, 9)  & 0.08 & 145 & 0.00  & 70  & 95  & 0,7     & &184 & 35,33 & \MaxT & &--  & --    & \MaxT \\
(60, 182, 9)  & 0.10 & 136 & 0.00  & 90  & 51  & 0,4     & &212 & 51,89 & \MaxT & &--  & --    & \MaxT \\
(60, 291, 9)  & 0.16 & 97  & 0.00  & 40  & 73  & 0,47    & &195 & 65,13 & \MaxT & &--  & --    & \MaxT \\
(60, 500, 9)  & 0.28 & 82  & 0.00  & 60  & 33  & 1,91    & &158 & 67,73 & \MaxT & &--  & --    & \MaxT \\
\hline
(65, 1000, 9) & 0.48 & 197 & 0.00  & 80  & 29  & 1913,95 & &186 & 69,9  & \MaxT & &--  & --    & \MaxT \\
(65, 1248, 9) & 0.60 & 195 & 0.00  & 100 & 29  & 851,48  & &244 & 77,05 & \MaxT& &--  & --    & \MaxT  \\
(65, 1352, 9) & 0.65 & 174 & 0.00  & 100 & 37  & 11,34   & &342 & 83,63 & \MaxT & &188 & 75    & \MaxT \\
(65, 728, 9)  & 0.35 & 101 & 0.00  & 90  & 28  & 2,5     & &165 & 66,07 & \MaxT& &--  & --    & \MaxT \\
(65, 832, 9)  & 0.40 & 151 & 19,21 & 100 & 30  & \MaxT & &161 & 65,22 & \MaxT  & &--  & --    & \MaxT \\
\hline
(70, 1000, 9) & 0.41 & 254 & 34,65 & 70  & 44  & \MaxT & &224 & 72,77 & \MaxT && --  & --    & \MaxT \\
(70, 1500, 9) & 0.62 & 270 & 0.00  & 70  & 39  & 1075,35 & &356 & 82,87 & \MaxT  & &--  & --    & \MaxT \\
(70, 248, 9)  & 0.10 & 149 & 0.00  & 60  & 100 & 1,09    & &210 & 50    & \MaxT && --  & --    & \MaxT \\
(70, 497, 9)  & 0.21 & 128 & 0.00  & 80  & 44  & 58,5    & &262 & 74,81 & \MaxT & &--  & --    & \MaxT \\
(70, 800, 9)  & 0.33 & 638 & 76,33 & 50  & 42  & \MaxT & &174 & 64,95 & \MaxT & &--  & --    & \MaxT
\\
\bottomrule
\end{tabular}
\\[0.5em]
\footnotetext
\flushleft
{Objective value (Obj), gap percentage (Gap), and time taken for each instance and configuration per model. The symbol {\MaxT} is reported whenever the maximum runtime {\MaxText} was reached. \hfill
}
\end{table}

All of the instances were solved up to optimality, except for the instance $(100, 1732, 6, 10)$ when using {\CGC}--B, reporting an optimality gap of $22,09\%$. Nevertheless, a feasible solution was still identified with this setup.
The lower objective values achieved through {\CGC}--A reflect its effectiveness in finding better columns than {\CGC}--B across all tested instances. This trend is also mirrored in the IP solution times where Gurobi was able to find optimal solutions faster by using {\CGC}--A, with the exception of the instance \((90, 1\,602, 5, 12)\) where {\CGC}--B stopped at an suboptimal solution. 
Notwithstanding, the differences in objective and solution times are not large, suggesting that both methods could be effectively used to assist the column generation heuristic. Consequently, {\CGC}--A will be the method of choice for subsequent experiments.

The last set of experiments \review{consists of} verifying the efficiency of the column generation approach assisted by spectral clustering against the exact MIP models. The performance of {\TercerModelo} \& {\CGC}, {\PrimerModelo} \& {\LBC-\SPC}, and {\SegundoModelo} \& {\LBC-\SPC} is compared with the aforementioned process. These experiments are presented in Table \ref{Table:ColGen}, which is organized as follows. The first column displays the instances identified by the 3--tuple $(n,|E|, k)$; the second column reports the graph density given by \( \nicefrac{2|E|}{n(n-1)} \);  columns 3 to 7 present the objective value, optimality gap percentage, the number of columns added in each iteration, the number of column generation iterations, and the running time in seconds, for the column generation approach; the latter experiment considers a maximum number of iterations fixed to \(100\), and the number of columns added corresponds to the value reporting the best objective in a range of 40 to 100 columns. Columns 8 to 10 present objective value, optimality gap percentage and running time in seconds for {\PrimerModelo} \& {\LBC-\SPC}. Finally, the last three columns report similar parameters for {\SegundoModelo} \& {\LBC-\SPC}. For all instances, the minimum partition size \(\alpha\) is set to \(2\).

In Table \ref{Table:ColGen}, the Column Generation approach combined with spectral clustering techniques demonstrates superior performance compared to configurations {\PrimerModelo} \& {\LBC-\SPC} and {\SegundoModelo} \& {\LBC-\SPC} in the majority of instances. Specifically, only in 5 out of 35 instances, {\TercerModelo} \& {\CGC}--A exhibits worse performance in terms of the objective value compared to the others. Regarding computational time, for instances of medium size (\(n < 65\)), the optimal solution of the model associated with {\TercerModelo} \& {\CGC}--A is found within a reasonable time that is bounded above by 153.62 seconds. For larger instances (\(n \geq 65\)), the maximum computational time is reached only in 3 out of 10 instances.

\review{In the context of Column Generation, pricing problems have traditionally been solved using combinatorial heuristics. However, the remarkable behavior of the Column Generation approach, when combined with spectral clustering, opens promising new methods for tackling complex partitioning or more general Combinatorial Optimization problems.
Furthermore, noticing that all instances were artificially-generated, the resulting experiments included a wide range of graph structures with varying densities. The extensive computational study shows that, regardless of the instance, the presented methodology is able to solve the problem efficiently. 
}

\section{Conclusions}
\label{conclusiones}

\review{In this study, we explored the graph partitioning problem in connected components, subject to minimum size constraints, aiming to partition an undirected graph into a predetermined number of subsets while minimizing the total edge cost. We presented novel Mixed Integer Programming formulations with flow constraints and a column generation technique enhanced by spectral clustering. The novelty arises from the integration of spectral clustering and spectral kernel learning techniques for solving the pricing problem within the column generation. This approach resulted in a high-performance algorithm that opens a promising path to solving complex Combinatorial Optimization problems. Additionally, the complexity of the pricing problem was derived and several families of valid inequalities were proved for each model.

The effectiveness of the formulations, valid inequalities, and solution techniques was thoroughly assessed with computational experiments that showed significant reductions in the number of B\&B nodes, optimality gaps, and overall solution times. Notably, the integration of column generation with spectral clustering and spectral kernel learning emerged as the most effective strategy, often outperforming the other MIP formulations. Furthermore, spectral clustering found better columns than spectral kernel learning when used to support the column generation algorithm.

A natural extension of the present study, and perhaps its limitation, is to derive general spectral-assisted methods that work adequately for other variants of the graph partitioning problem. This could lead to more robust solutions for a wider range of problems.

Finally, the extension of spectral-assisted heuristics for other Combinatorial Optimization problems should be further explored. This work demonstrated the strengths of employing advanced clustering techniques and traditional optimization methods in a single framework, paving the way for future research in this direction.
}


\section*{\normalsize{DATA, CODE, AND MATERIALS}}
Code for the experiments, implemented models, and instance generation is available at 

\href{https://github.com/Particionamiento/GraphPartitioningMinSize}{\texttt{https://github.com/Particionamiento/GraphPartitioningMinSize}}


\section*{\normalsize{STATEMENTS AND DECLARATIONS}}

\noindent\textbf{Competing Interests}
The authors declare that there is no conflict of interest.

\vspace{0.75\baselineskip}
\noindent\textbf{Funding}
M.C. was partially supported by Escuela Politécnica Nacional's Master Program in Mathematical Optimization Scholarship.
A.M-T. acknowledges support of MAC-MIGS CDT Scholarship under EPSRC grant EP/S023291/1. 
D.R., R.T., and P.V. acknowledge the support by Escuela Politécnica National Research Project \texttt{PII-DM-2019-03} and \texttt{PII-DM-2020-02}.

\vspace{0.75\baselineskip}
\noindent\textbf{Acknowledgements} A.M-T. would like to thank Dr John Pearson (The University of Edinburgh) for their enriching comments and discussions during the preparation of this draft.

\bibliography{Main_plain_2.bbl}

\end{document}